\documentclass[a4paper,10pt,reqno, english]{amsart}

\usepackage{amsmath,amssymb,amscd,amsthm,amsfonts}
\usepackage{graphicx,subfigure}
\usepackage{hyperref}
\usepackage[capitalize,noabbrev]{cleveref}
\usepackage{dsfont}
\usepackage{xcolor}
\usepackage[utf8]{inputenc}
\usepackage[nobysame,non-compressed-cites]{amsrefs}

\newtheorem{theorem}{Theorem}[section]

\newtheorem{lemma}[theorem]{Lemma}

\newtheorem{claim}[theorem]{Claim}
\newtheorem{corollary}[theorem]{Corollary}

\newtheorem{conjecture}[theorem]{Conjecture}

\theoremstyle{definition}
\newtheorem{definition}[theorem]{Definition}

\numberwithin{equation}{section}

\newcommand{\rr}{\mathds{R}}
\newcommand{\zz}{\mathds{Z}}
\newcommand{\ff}{\mathcal{F}}

\DeclareMathOperator{\conv}{conv}
\DeclareMathOperator{\diam}{diam}

\title{A m\'elange of diameter Helly-type theorems}

\author[Dillon]{Travis Dillon}
\address{Lawrence University, 711 E.~Boldt Way, Appleton, WI 54911}
\email{travis.a.dillon@lawrence.edu}

\author[Sober\'on]{Pablo Sober\'on}\address{Baruch College, City University of New York, One Bernard Baruch Way, New York, NY 10010} 
\email{pablo.soberon-bravo@baruch.cuny.edu}

\thanks{This research project was completed as part of the 2020 Baruch Discrete Mathematics REU, supported by NSF awards DMS-1802059, DMS-1851420, and DMS-1953141.  Sober\'on's research is also supported by PSC-CUNY grant 62639-00-50.}

\begin{document}

\begin{abstract}
A Helly-type theorem for diameter provides a bound on the diameter of the intersection of a finite family of convex sets in $\mathds{R}^d$ given some information on the diameter of the intersection of all sufficiently small subfamilies.  We prove fractional and colorful versions of a longstanding conjecture by B\'ar\'any, Katchalski, and Pach.  We also show that a Minkowski norm admits an exact Helly-type theorem for diameter if and only if its unit ball is a polytope and prove a colorful version for those that do.  Finally, we prove Helly-type theorems for the property of ``containing $k$ colinear integer points.''
\end{abstract}

\maketitle

\section{Introduction}

Helly's theorem is one of the most prominent results on the intersection properties of families of convex sets \cites{Radon:1921vh, Helly:1923wr}.  It says that  \textit{if the intersection of every $d+1$ or fewer elements of a finite family of convex sets in $\rr^d$ is nonempty, then the intersection of the entire family is nonempty.}  This result has many extensions and generalizations, including topological, colorful, and fractional variants (see, for example, \cites{Amenta:2017ed, Holmsen:2017uf} and the references therein).

Quantitative versions of Helly's theorem guarantee that the intersection of a family of convex sets is not just nonempty but ``large'' in some quantifiable sense. B\'ar\'any, Katchalski, and Pach initiated this direction of research \cites{Barany:1982ga, Barany:1984ed} when they proved that \textit{if the intersection of every $2d$ or fewer elements of a finite family of convex sets in $\rr^d$ has volume greater than or equal to 1, then the intersection of the entire family has volume at least $d^{-2d^2}$.}

Nasz\'odi \cite{Naszodi:2016he} improved the guarantee of the volume in the intersection to $d^{-2d}$, mostly settling this volumetric variant. His approach, based on sparsification of John decompositions of the identity, has been improved in several articles \cites{Brazitikos:2016ja, Brazitikos:2017ts, Damasdi:2019vm, Vidal:2020iw}. A constellation of related results that adjust the function measuring the size of the intersection, the cardinality of the subfamily intersection, or the guarantee in the conclusion have since been proven \cites{DeLoera:2017gt, Soberon:2016co, Rolnick:2017cm, Sarkar:2019tp}. B\'ar\'any, Katchalski, and Pach conjectured a Helly-type theorem for diameter as well, which remains open.

\begin{conjecture}[B\'ar\'any, Katchalski, Pach 1982 \cite{Barany:1982ga}]\label{conj:BKP}
	  Let $\mathcal{F}$ be a finite family of convex sets in $\rr^d$.  If the  intersection of every $2d$ or fewer members of $\ff$ has diameter greater than or equal to 1, then the intersection of $\ff$ has diameter greater than or equal to $c d^{-1/2}$, for some absolute constant $c > 0$.
\end{conjecture}

B\'ar\'any, Katchaslki, and Pach showed that the diameter of the intersection is at least $d^{-2d}$.  Brazitikos improved this to $d^{-11/2}$, which is the first polynomial bound on the diameter of the intersection \cite{Brazitikos:2018uc}.  Brazitikos \cite{Brazitikos:2016ja} also proved, under the hypothesis that the intersections of subfamilies of size $\alpha d$ have diameter at least 1 (for some large enough absolute constant $\alpha$), that the intersection of the entire family has diameter at least $c d^{-3/2}$ and strengthened the bound to $c d^{-1/2}$ under the further assumption that each set is centrally symmetric. Asymptotically optimal bounds are known for much larger subfamily intersection sizes \cite{Soberon:2016co}.

In this manuscript we prove several new Helly-type theorems for the diameter. First, we prove that Conjecture \ref{conj:BKP} holds for at least a large subfamily of $\ff$.

\begin{theorem}\label{thm:diameter-fractional}
	There exists a decreasing function $\gamma\colon (0,\sqrt{2}) \to (0,1]$ such that $\gamma(c) \to 1$ as $c \to 0$ and the following holds for every $c \in (0,\sqrt{2})$, $\alpha \in (0,1]$, and $d \geq 2$.  Let $\beta = 1- \big(1-\alpha\cdot \gamma(c)\big)^{1/2d}$ and $\ff$ be a finite family of convex sets in $\rr^d$.  If $\bigcap \mathcal H$ has diameter greater than or equal to 1 for at least $\alpha \binom{|\ff|}{2d}$ subcollections $\mathcal H \subseteq \ff$ of $2d$ sets, then there exists a subfamily $\mathcal G \subseteq \ff$ such that $|\mathcal G| \ge \beta |\ff|$ and the diameter of $\bigcap \mathcal G$ is greater than or equal to $c d^{-1/2}$.
\end{theorem}

We exclude the case $d=1$ since the real line has an exact diameter Helly-type theorem, while all higher dimensions do not. Crucially, for $\alpha=1$ we have $\beta \to 1$ as $c \to 0$, so the size of the subfamily can be arbitrarily close to that of the original set.  \cref{thm:diameter-fractional} suggests that it should be possible to extend the diameter conjecture to a fractional version (with an absolute $c$) in the same way that Katchaslki and Liu generalized Helly's theorem \cite{Katchalski:1979bq}.

We also prove a colorful variant of Theorem \ref{thm:diameter-fractional}, similar to Lov\'asz's colorful Helly theorem \cite{Barany:1982va}, at the cost of a slightly smaller constant $\beta$.

\begin{theorem}\label{thm:diameter-fractional-colorful}
Let $\gamma\colon (0,\sqrt{2}) \to (0,1]$ be the function in Theorem \ref{thm:diameter-fractional}. For each $c \in (0,\sqrt{2}$), $\alpha \in (0,1]$, and $d \geq 2$, set $\beta = 1-2d\big(1-\alpha\cdot \gamma(c)\big)^{1/2d}$. Assume that $\ff_1, \ldots, \ff_{2d}$ are finite families of convex sets in $\rr^d$ and set $N = \prod_{i=1}^{2d} |\ff_i|$.  If $\bigcap_{i=1}^{2d}F_i$ has diameter greater than or equal to 1 for at least $\alpha \binom{N}{2d}$ different $2d$-tuples $(F_i)_{i=1}^{2d}$ with $F_i \in \ff_i$ for each $i$, then there exists an index $k\in [2d]$ and a subfamily $\mathcal G \subseteq \ff_k$ such that $|\mathcal G| \ge \beta |\ff_k|$ and the diameter of $\bigcap \mathcal G$ is greater than or equal to $c d^{-1/2}$.
\end{theorem}

The description ``colorful'' is derived from thinking of each family $\ff_i$ as having a particular color.  Then, if the intersections of sufficiently many colorful collections (containing one set of each color) have large diameter, there is a large monochromatic family whose intersection has large diameter. In light of \cref{thm:diameter-fractional-colorful}, we postulate a colorful version of the B\'ar\'any-Katchalski-Pach conjecture.

\begin{conjecture}
	Let $\ff_1, \ldots, \ff_{2d}$ be finite families of convex sets in $\rr^d$.  If $\bigcap_{i=1}^{2d} F_i$ has diameter greater than or equal to 1 for every $2d$-tuple $(F_i)_{i=1}^{2d}$ with $F_i \in \ff_i$, then there exists an index $k \in [2d]$ such that $\bigcap \ff_k$ has diameter greater than or equal to $c d^{-1/2}$, for some absolute constant $c$.
\end{conjecture}

Any Helly-type theorem for diameter necessarily entails some loss---it is not possible to conclude that intersection of the entire family has diameter at least 1 even by checking arbitrarily large subfamilies \cite{Soberon:2016co}.  Sarkar, Xue, and Sober\'on \cite{Sarkar:2019tp} suggested that this may be a consequence of the norm used to measure diameter, and that the $\ell_1$ norm may give exact Helly-type diameter results.  We show that this is indeed the case.

\begin{theorem}\label{thm:Minkowski}
	Let $\rho$ be a Minkowski norm in $\rr^d$ whose unit ball is a polytope with $k$ facets, and let $\ff$ be a finite family of convex sets in $\rr^d$.  If the intersection of every $kd$ or fewer members of $\ff$ has $\rho$-diameter greater than or equal to 1, then $\bigcap \ff$ has $\rho$-diameter greater than or equal to 1.  Moreover, this statement is not true if $kd$ is replaced by $kd-1$.
\end{theorem}

In particular, there is an exact diameter Helly-type theorem in the $\ell_1$-norm, although the intersection condition on subfamilies is necessarily exponential. We present three proofs of \cref{thm:Minkowski}, one of which implies a colorful version (see Theorem \ref{thm:Minkowski-colorful}).  In \cref{thm:Minkowski-not-polytope}, we prove that no other Minkowski norm admits an exact Helly-type theorem for diameter, thus characterizing the norms for which an exact theorem is possible.

The particular case of the $\ell_\infty$ norm implies a different relaxation of Conjecture \ref{conj:BKP}.

\begin{corollary}
	Let $\ff$ be a finite family of convex sets in $\rr^d$.  If the intersection of every $2d^2$ or fewer elements of $\ff$ has diameter greater than or equal to 1, then $\bigcap \ff$ has diameter greater than or equal to $d^{-1/2}$.
\end{corollary}

If we relax \cref{conj:BKP} to checking subfamilies of quadratic cardinality in the dimension, then an application of N\'aszodi's method guarantees that the diameter of $\bigcap \ff$ is at least $d^{-1}$ (see, e.g., \cite{Vidal:2020iw}*{Theorem 1.4}). To obtain a bound of $d^{-1/2}$, however, the method would require that each set be centrally symmetric.

Finally, we investigate a discrete analogue of diameter Helly-type theorems.  Doignon extended Helly's theorem to the integer lattice \cite{Doignon:1973ht}, showing that \textit{if the intersection of every $2^d$ or fewer elements of a finite family of convex sets in $\rr^d$ contains an integer point, then the entire intersection also contains an integer point.} This result was proved independently by Bell \cite{Bell:1977tm} and by Scarf \cite{Scarf:1977va}.  In most cases, the aim of quantitative Helly-type theorems for the integer lattice is to bound the number of integer points in the intersection of a family of convex sets \cites{Aliev:2016il, Averkov:2017ge, DeLoera:2017bl, Dillon:2020ab}.

Such work can be thought of as Helly-type theorems for ``discrete volume''\!\!\!. We think of a convex set as having large ``discrete diameter'' if it contains many colinear integer points. In contrast to most continuous diameters, there is an exact Helly-type theorem for discrete diameter.

\begin{theorem}\label{thm:diameter-discrete}
	Let $k$ be a positive integer and $\ff$ be a finite family of convex sets in $\rr^d$.  If the intersection of every $4^d$ or fewer elements of $\ff$ contains $k$ colinear integer points, then $\bigcap \ff$ contains $k$ colinear integer points.
\end{theorem}

Our proof also implies a colorful version of Theorem \ref{thm:diameter-discrete}. Doignon's theorem shows that the size of the subfamilies in the hypothesis is necessarily exponential in the dimension, but this size can be significantly reduced if it suffices to maintain a bound on the diameter of a large subfamily of $\mathcal F$.

\begin{corollary}\label{cor:discrete-fractional}
	For every positive integer $d$ and real number $\alpha \in (0,1]$, there exists a real number $\beta = \beta(\alpha, d) > 0$ such that the following holds.  Assume that $\ff$ is a finite family of convex sets in $\rr^d$ and let $k$ be a positive integer. If $\bigcap \mathcal H$ contains at least $k$ colinear integer points for at least $\alpha \binom{|\ff|}{2d+1}$ subcollections $\mathcal H \subseteq \mathcal F$ of $2d+1$ sets, then there exists a subfamily $\mathcal G \subseteq \ff$ such that $|\mathcal G| \ge \beta |\ff|$ whose intersection contains $k$ colinear integer points.
\end{corollary}

Since Doignon's theorem is optimal, the proportion $\beta(1,d)$ is necessarily strictly less than 1.

We present the proof of Theorem \ref{thm:diameter-discrete} in Section \ref{sec:discrete}.  Our results for Minkowski norms are collected in Section \ref{sec:minkowski}, and \cref{sec:frac} contains the proofs of theorems \ref{thm:diameter-fractional} and \ref{thm:diameter-fractional-colorful}.

\section{Discrete diameter results}\label{sec:discrete}

The proofs in this section employ similar methods to those of Sarkar, Xue, and Sober\'on in \cite{Sarkar:2019tp}, in which a suitable parametrization reduces quantitative Helly-type theorems to standard Helly-type theorems in higher-dimensional spaces.

We denote the standard inner product in $\rr^d$ by $\langle \cdot, \cdot \rangle$.

\begin{proof}[Proof of Theorem \ref{thm:diameter-discrete}]
Let $v \in \rr^d$ be a vector whose components are algebraically independent.  In particular, $\langle v, z \rangle \neq 0$ for every $z \in \zz^d \setminus \{0\}$.  For every convex set $K \subseteq \rr^d$, we define the set
	\[
	S(K) =\big\{(x,y) \in \rr^d \times \rr^d : x \in K, \ x+(k-1)y \in K, \ \langle v, y \rangle > 0\big\},
	\]
which is convex.  Moreover, if $x \in K$, $x + (k-1)y \in K$, and $\langle y, v \rangle \neq 0$, then either $(x, y) \in S(K)$ or $(x+(k-1)y, -y) \in S(K)$.  Now consider the family
\[
    \mathcal{G} = \{S(K) : K \in \ff\}.
\]
The conditions of the theorem imply that the intersection of every $2^{2d}$ or fewer sets in $\mathcal{G}$ contains a point of $\zz^{2d}$ in their intersection.  By Doignon's theorem, $\bigcap \mathcal{G}$ contains an integer point.  If $(x, y)$ is such a point, then $y \neq 0$ and the $k$ colinear integer points $x, x+y, \ldots, x+(k-1)y$ are contained in every member of $\ff$.
\end{proof}

The proof above is quite malleable.  For example, replacing Doignon's theorem with its colorful version (proved by De Loera, La Haye, Oliveros, and Rold\'an-Pensado \cite{DeLoera:2017th}) yields a colorful version of Theorem \ref{thm:diameter-discrete}.  To obtain Corollary \ref{cor:discrete-fractional}, we replace Doignon's theorem by the following fractional version.

\begin{theorem}[B\'ar\'any, Matou\v{s}ek 2003 \cite{Barany:2003wg}]
    For every positive integer $d$ and real number $\alpha \in (0,1]$ there exists a real number $\beta = \beta(\alpha,d) > 0$ such that the following is true. If $\mathcal F$ is a finite family of convex sets such that $\bigcap \mathcal H$ contains an integer point for at least $\alpha \binom{\lvert \mathcal F\rvert}{d+1}$ subcollections $\mathcal H \subseteq \mathcal F$ of $d+1$ sets, then there is a subfamily $\mathcal G \subseteq \mathcal F$ with at least $\beta \lvert \mathcal F \rvert$ sets whose intersection contains an integer point.
\end{theorem}

It is unclear whether the number $4^d$ in \cref{thm:diameter-discrete} is optimal. Since the case $k=1$ is Doignon's theorem, $4^d$ cannot be replaced by anything smaller than $2^d$. The following construction, which generalizes a construction communicated by Gennadiy Averkov for $d=2$, improves the lower bound to $d2^d$.

\begin{claim}
For each $d \geq 1$, there exists a finite family $\mathcal F$ of convex sets in $\rr^d$ such that the intersection of any $d2^d-1$ sets in $\mathcal F$ contains 3 colinear integer points but $\bigcap \mathcal F$ does not.
\end{claim}
\begin{proof}
Let $R \subseteq \{0,1,2,3\}^d$ be the collection of integer points where exactly one coordinate is in $\{0,3\}$ and let $Q = \{1,2\}^d$. That is, $Q \cup R$ is the set of integer points in the hypercube $[0,3]^d$ that do not lie on its $(d-2)$-skeleton. We define 
\[  \mathcal F = \{\conv(Q \cup R \setminus \{x\}) : x \in R \}, \]
which is a collection of $d2^d$ sets. The intersection of any $d2^d - 1$ sets in $\mathcal F$ contains every point in $Q$ and one point in $R$, so it contains 3 colinear integer points. But the integer points in $\bigcap \mathcal F$ are exactly those in $Q$, which does not contain 3 colinear integer points.
\end{proof}

\section{Diameter results for Minkowski norms}\label{sec:minkowski}

\begin{definition}
    Let $K \subseteq \rr^d$ be a compact convex set with nonempty interior that is symmetric about the origin. The \textit{Minkowski norm} of $K$ is defined by\vspace{-0.4em}
    \[  \rho_K(x) = \min\{ t \geq 0 : x \in tK\}.  \]
\end{definition}

Any Minkowski norm is a norm in the usual sense. Given a vector $v \in \rr^d \setminus \{0\}$, the \textit{$v$-width} of a compact convex set $K \subseteq \rr^d$ is $\max_{x,y\in K} \langle x-y, v\rangle$. Let $K \subseteq \rr^d$ polytope with $k$ facets, and for each facet $L_i$, let $v_i$ be the vector such that $\{ x \in \rr^d : \langle x, v_i \rangle = 1\}$ is the hyperplane containing $L_i$. We assume that $L_i$ and $L_{(k/2)+i}$ are opposing facets, so $v_{(k/2)+i} = -v_i$. If $\rho$ is the Minkowski norm whose unit ball is $K$, then
\begin{equation}\label{eq:polytope-v-width}
    \rho(x) = \max_{1 \leq i \leq k/2} \lvert \langle x,v_i\rangle \rvert
\end{equation}
for every $x \in \rr^d$.

The proof in \cref{sec:discrete} can be adapted to simplify the proof of a Helly-type theorem for $v$-width by the second author \cite{Soberon:2016co}.

\begin{theorem}\label{thm:v-width}
	Let $v$ be a nonzero vector in $\rr^d$ and $\ff$ be a finite family of compact convex sets in $\rr^d$.  If the intersection of every $2d$ sets in $\ff$ has $v$-width greater than or equal to 1, then $\bigcap \ff$ has $v$-width greater than or equal to 1.
\end{theorem}

\begin{proof}
	For every convex set $K \subseteq \rr^d$, let
	\[
	S(K) =\big\{(x,y) \in \rr^d \times \rr^d : x \in K, \ x+y \in K, \ \langle y, v \rangle = 1\big\}.
	\]
	The set $S(K)$ is convex.  Consider the set $\mathcal{G} = \{S(K) : K \in \ff\}$; this family is contained in an affine subspace of dimension $2d-1$ by the condition $\langle y, v\rangle = 1$.  The hypothesis of the theorem implies that every $2d$ or fewer elements of $\mathcal{G}$ intersect, so $\bigcap \mathcal G$ is nonempty by Helly's theorem. If $(x,y) \in \bigcap \mathcal G$, then $x,\, x+y \in F$ for every set $F \in \mathcal F$, which shows that the $v$-width of $\bigcap \mathcal F$ is at least 1.
\end{proof}

As before, substituting fractional or colorful versions of Helly's theorem in the proof provides corresponding variants of \cref{thm:v-width}. In fact, we can substitute Helly's theorem with Kalai and Meshulam's topological extension of the colorful Helly theorem \cite{Kalai:2005tb}, which generalizes the coloring structure to an arbitrary matroid. Such an extension is much stronger than the $v$-width results obtained in \cite{Soberon:2016co}.  For later reference, we explicitly state the fractional version of \cref{thm:v-width} derived from the fractional Helly theorem (first proved by Katchalski and Liu \cite{Katchalski:1979bq}).

\begin{theorem}\label{thm:v-width-fractional}
    For every positive integer $d$ and real number $\alpha \in (0,1]$, there exists a real number $\beta = \beta(\alpha, d) > 0$ such that the following holds for every nonzero vector $v \in \rr^d$. If $\bigcap \mathcal H$ contains has $v$-width greater than or equal to 1 for at least $\alpha \binom{|\ff|}{2d}$ subcollections $\mathcal H \subseteq \mathcal F$ of $2d$ sets, then there exists a subfamily $\mathcal G \subseteq \ff$ such that $|\mathcal G| \ge \beta |\ff|$ whose intersection has $v$-width greater than or equal to 1.
\end{theorem}

Even though the theorem above is already known \cite{Soberon:2016co}, this method allows us to deduce the explicit bound $\beta(\alpha, d) \ge 1-(1-\alpha)^{1/2d}$ by using the results of Kalai \cites{Kalai:1984isa, Kalai:1986ho}. We now use the Helly-type theorem for $v$-width to prove a colorful version of \cref{thm:Minkowski}.

\begin{theorem}\label{thm:Minkowski-colorful}
    Let $\rho$ be a Minkowski norm in $\rr^d$ whose unit ball is a polytope with $k$ facets, and let $\ff_1, \ldots, \ff_{kd}$ be finite families of convex sets in $\rr^d$.  If $\bigcap_{i=1}^{2d}F_i$ has $\rho$-diameter at least $1$ for every $kd$-tuple $(F_i)_{i=1}^{2d}$ such that $F_i \in \ff_i$ for each $i$, then there exists an index $l \in [kd]$ such that $\bigcap \ff_l$ has $\rho$-diameter at least $1$. Moreover, the same statement is not true if $kd$ is replaced by $kd-1$.
\end{theorem}

\begin{proof}
We prove the contrapositive. Assume that $\ff_1, \dots, \ff_{kd}$ are finite families of convex sets in $\rr^d$ such that $\bigcap \ff_i$ has $\rho$-diameter at most 1 for each $i \in [kd]$. We want to find a colorful $kd$-tuple whose intersection has $\rho$-diameter at most 1.

For each facet $L_j$ of $P$, let $v_j$ be the vector in $\rr^d$ such that $\langle x,v\rangle = 1$ for every $x \in L_j$. We choose a labelling of the facets so that $L_{(k/2)+j} = -L_j$ for each $j \in [k/2]$. From the assumption on $\rho$-width of $\ff_i$ and \eqref{eq:polytope-v-width}, the $v_j$-width of $\bigcap \ff_i$ is at most 1 for each $i \in [kd]$ and $j \in [k/2]$. Applying the contrapositive of the colorful version of Theorem \ref{thm:v-width} to the collection of $2d$ families $\mathcal F_{2(j-1)d+1},\dots, \mathcal F_{2jd}$ implies that there is a set $F_{i+2(j-1)d} \in \mathcal F_{i + 2(j-1)d}$ for each $i\in [2d]$ such that $\bigcap_{i=1}^{2d} F_{2(j-1)d+i}$ has $v_j$-width less than or equal to 1.

Let $\mathcal G$ denote the family $\{F_1,\dots,F_{kd}\}$.  By construction, $\mathcal{G}$ has exactly one element from each $\ff_i$. Its intersection $\bigcap \mathcal G$ has $v_j$-width at most 1 for every $j \in [k/2]$, so the $\rho$-diameter of $\bigcap \mathcal G$ is at most 1 by \eqref{eq:polytope-v-width}.

Now we prove optimality. Consider a set $\{x_1, \ldots, x_k\}$ of points in $\rr^d$ such that $x_i$ is in the relative interior of $L_i$ and $x_i = -x_{(k/2)+i}$.  For each $i \in [k/2]$, choose $d$ closed half-spaces such that
\begin{itemize}
	\item each half-space contains every point in $\{x_j\}_{j \neq i}$,
	\item the intersection of the $d$ half-spaces with $L_i$ is the singleton $\{x_i\}$, and
	\item the intersection of any $d-1$ of them contains a point $y$ such that $\langle y, v_i \rangle > 1$.
\end{itemize}

Let $\mathcal F$ be the collection all $kd$ half-spaces. The intersection $\bigcap \ff$ is contained in $P$, so its $\rho$-diameter is at most $2$.  For any subset $\ff' \subseteq \ff$ of size $kd-1$, there is a facet $L_i$ of $P$ with at most $d-1$ corresponding half-spaces in $\ff'$.  Therefore, there exists a point $\tilde{x} \in \bigcap \ff'$ outside of $P$ such that the segment $0\tilde{x}$ intersects $L_i$.  We can choose $\tilde{x}$ close enough to $x_i$ so that the segment between $-x_i$ and $\tilde{x}$ has $\rho$-length greater than $2$.  Therefore, the $\rho$-diameter $\bigcap \mathcal F'$ is greater than $2$ for every subset $\mathcal F' \subseteq \mathcal F$ of size $kd-1$, while the $\rho$-diameter of $\bigcap \ff$ is at most $2$. Taking $\ff_i = \ff$ for each $i \in [kd-1]$ shows the optimality of the parameter $kd$.
\end{proof}

Setting $\ff_i = \ff$ for each $i \in [kd]$ in \cref{thm:Minkowski-colorful} proves the Helly-type statement in Theorem \ref{thm:Minkowski}, and the proof the optimality of $kd$ also carries over to the monochromatic version.

\cref{thm:Minkowski} implies the following more general version of Corollary \ref{cor:discrete-fractional}.

\begin{theorem}\label{thm:lp}
	Let $p \geq 1$ and $\ff$ be a finite family of convex sets in $\rr^d$.  If the intersection of every $2d^2$ or fewer sets in $\ff$ has $\ell_p$-diameter greater than or equal to 1, then $\bigcap \ff$ has $\ell_p$-diameter greater than or equal to $d^{-1/p}$. 
\end{theorem}

\begin{proof}
	A set in $\rr^d$ with $\ell_p$-diameter at least 1 has $\ell_\infty$-diameter at least $d^{-1/p}$. Since the unit ball in the $\ell_{\infty}$ norm is a polytope with $2d$ facets, we can employ Theorem \ref{thm:Minkowski} to conclude that the $\ell_{\infty}$-diameter of $\bigcap \ff$ is at least $d^{-1/p}$. The $\ell_p$-diameter of $\ff$ is at least $d^{-1/p}$ as well.
\end{proof}

The $\ell_1$ norm is a useful lens with which to compare our results.  Theorem \ref{thm:lp} says that we can bound the $\ell_1$-diameter of $\bigcap \ff$ by $d^{-1}$ if we know that the intersection of every $2d^2$ sets in $\mathcal F$ has $\ell_1$-diameter greater than or equal to 1, whereas Theorem \ref{thm:Minkowski} says that we can bound the $\ell_1$-diameter of $\bigcap \ff$ by $1$ if we know that the intersection of every $d 2^d$ in $\mathcal F$ sets has $\ell_1$-diameter greater than or equal to 1.  Neither of these consequences implies the other.

We now present two additional proofs of the Helly-type statement in Theorem \ref{thm:Minkowski}. The first proof uses the following lemma, in which the boundary of a set $K \subseteq \rr^d$ is denoted by $\partial K$. 

\begin{lemma}\label{lem:helly-polytope}
    Let $P \subseteq \rr^d$ be a centrally symmetric polytope with $k$ facets and $\mathcal{G}$ be a finite family of sets in $\rr^{2d}$ such that
    \begin{enumerate}
        \item $K \cap (\rr^d \times L)$ is convex for every facet $L$ of $P$ and every $K \in \mathcal{G}$, and
        \item if $x,y \in \rr^d$, $K \in \mathcal G$, and $(x,y) \in K$, then $(x+y,-y) \in K$.
    \end{enumerate}
    If the intersection of every $kd$ or fewer sets in $\mathcal{G}$ contains a point in $\rr^d \times \partial P$, then $\bigcap \mathcal{G}$ contains a point in $\rr^d \times \partial P$.
\end{lemma}

\begin{proof}
	The general approach is similar to that of Radon's proof of Helly's theorem \cite{Radon:1921vh}. We proceed by induction on $\lvert \mathcal G \rvert$.  If $\lvert \mathcal G\rvert\leq kd$, there is nothing to prove, so we assume the result holds for for all collections of convex sets with $n \ge kd$ members.  Let $G = \{K_1,\dots,K_{n+1}\}$ be a collection of $n+1$ convex sets in $\rr^{2d}$ that satisfies the hypothesis of the lemma.  The induction hypothesis implies that for each $i \in [n+1]$ there is a point $(x_i,y_i) \in \bigcap (\mathcal{G} \setminus K_i)$ such that $y_i \in \partial P$.  Grouping the facets of $P$ by opposing pairs,  there must be a pair of facets $L, -L$ whose union contains at least
	\[
	    \frac{n+1}{k/2} > 2d
	\]
	points in $\{y_i\}_{i=1}^{n+1}$.  By replacing $(x_i, y_i)$ by $(x_i+y_i, -y_i)$ if necessary, the facet $L$ contains at least $2d+1$ points in $\{y_i\}_{i=1}^{n+1}$.  Therefore $\rr^d \times L$ contains at least $2d+1$ points in $\{(x_i,y_i)\}_{i=1}^{n+1}$. Since $\rr^d \times L$ is a $(2d-1)$-dimensional affine subspace, applying Radon's lemma to these $2d+1$ points yields a partition of them into two sets $A$ and $B$ whose convex hulls intersect.  Any point in $\conv (A) \cap \conv (B)$ is in $\bigcap \mathcal{G}$ as well as $\rr^d \times L \subseteq \rr^d \times \partial P$. 
\end{proof}

\begin{proof}[Second proof of Theorem \ref{thm:Minkowski}]
	Let $P$ be the unit ball of $\rho$.  For each convex set $K \subseteq \rr^d$, let\vspace{-0.3em}
	\[
	    S(K) = \big\{ (x,y) \in \rr^d \times \rr^d : x\in K, \ x+y \in K, \ \rho(y) = 1 \big\}.
	\]
	The conditions of Theorem \ref{thm:Minkowski} ensure that we can apply Lemma \ref{lem:helly-polytope} to the family $\mathcal{G} = \{S(K) : K \in \ff\}$.  Given a point $(x,y) \in \bigcap \mathcal{G} \cap (\rr^d \times \partial P)$, every set in $\ff$ contains the segment between $x$ and $x+y$; since $\rho(y) = 1$, the $\rho$-diameter of $\bigcap \mathcal G$ is at least 1.
\end{proof}

Our final proof of \cref{thm:Minkowski} relies on a limit argument.

\begin{proof}[Third proof of Theorem \ref{thm:Minkowski}]
We may assume without loss of generality that every set in $\mathcal F$ is compact. We denote by $K^n$ the $n$-fold product of the set $K$.  Let $f\colon (\rr^d)^k \to \rr$ be defined by
\[
	f(x_1, y_1, \ldots, x_{k/2}, y_{k/2}) = \sum_{i=1}^{k/2} \langle y_i-x_i, v_i\rangle.
\]

If a set $K \subseteq \rr^d$ has $\rho$-diameter greater than or equal to 1, then it has $v_i$-width at least 1 for some $i \in [k/2]$. Therefore there exists a point  $\bar{x} \in K^{k}$ such that $f(\bar{x}) \geq 1$. For each $K \in \mathcal F$, consider the set
\[
    S(K) = \{\bar{x} \in K^{k}: f(\bar{x}) = 1\},
\]
which is convex and lies in an affine subspace of dimension $kd-1$.  Therefore, an application of Helly's theorem implies that $(\bigcap \mathcal F)^{k}$ contains a point $\bar{a}_1$ with $f(\bar{a}_1) = 1$.

Now consider the function $g\colon (\rr^d)^k \to \rr$ defined by
\[
    g(x_1, y_1, \ldots, x_{k/2}, y_{k/2}) = \max_{1\le i\le k/2} \langle y_i - x_i, v_i\rangle.
\]
If $g(\bar{a}_1) = 1$, we are done.  Otherwise, equation \eqref{eq:polytope-v-width} implies that that for each $kd$-tuple $\mathcal F'\subseteq \mathcal F$, there are two points $x,y \in \bigcap \mathcal F'$ and an index $i \in [k]$ such that $\langle y - x, v_i\rangle \ge 1$.  Replacing the corresponding coordinates of $\bar{a}_1$ by $x, y$, we obtain a new point $\bar{x} \in  \left( \bigcap \mathcal F'\right)^{k}$ such that $f(\bar{x}) \ge 1 + (1-g(\bar{a}_1))$.

Bootstrapping the previous arguments, we can find a point $\bar{a}_2 \in \left( \bigcap \mathcal F\right)^{k}$ such that $f(\bar{a}_2) = 2 - g(\bar{a}_1)$.  Iterating this argument creates a sequence $(\bar{a}_n)_{n=1}^\infty$ in $\left( \bigcap \mathcal F\right)^{k}$ such that
\[
    f(\bar{a}_n) \ge {n - \sum_{i=1}^{n-1}g(\bar{a}_i)}
\]
for each $n \in \mathds{N}$. Let $\beta_n = \max\{g(\bar{a}_1), \ldots, g(\bar{a}_n)\}$. We have
\[
	\beta_n \ge g(\bar{a}_n)
	\ge \frac{f(\bar{a}_n)}{d}
	\ge \frac{n - \sum_{i=1}^{n-1}g(\bar{a}_i)}{d}
	\ge \frac{n - (n-1)\beta_n}{d}.
\]
In other words, the $\rho$-diameter of $\bigcap \mathcal F$ is at least $\beta_n \geq n/(d+n-1)$ for every $n \ge 1$.  Taking the limit as $n \to \infty$ finishes the proof.
\end{proof}

The next result shows that there is no exact Helly-type theorem for diameter for any norm whose unit ball is not a polytope.

\begin{theorem}\label{thm:Minkowski-not-polytope}
Let $\rho$ be a Minkowski norm in $\mathds{R}^d$ whose unit ball is not a polytope.  Then, for every integer $n$ there exists a finite family $\mathcal{G}$ of convex sets such that the intersection of every $n$ or fewer sets in $\mathcal G$ has $\rho$-diameter greater than or equal to $1$, but the $\rho$-diameter of $\bigcap \mathcal G$ is strictly less than $1$.
\end{theorem}

\begin{proof}
Let $P$ be the unit ball of $\rho$ and $\mathcal{F}$ be the infinite family of closed containment-minimal half-spaces that contain $P$.  We parametrize $\mathcal{F}$ using the unit sphere $S^{d-1}$ by associating each vector $x \in S^{d-1}$ with the half-space $H_x \in \mathcal F$ whose bounding hyperplane is perpendicular to it and that contains an infinite ray in the direction of $x$. For any finite family $\ff' \subseteq \ff$, the unit ball $P$ is strictly contained in $\bigcap \ff'$, so the $\rho$-diameter of $\bigcap \ff'$ is strictly larger than $2$. 

We define a function $f\colon (S^{d-1})^n \to \rr$ by
\[
    f(x_1, \ldots, x_n) = \min \bigg\lbrace 100,\ \operatorname{diam}_{\rho} \Big(\bigcap_{i=1}^n H_{x_i}\Big)\bigg\rbrace.
\]
The minimum ensures that $f$ is well-defined when $\bigcap_{i=1}^n H_{x_i}$ is unbounded.  The function $f$ is continuous, and $f(x_1,\ldots, x_n)>2$ for every $n$-tuple in $(S^{d-1})^n$.  Since the domain of $f$ is compact, $f$ attains a minimum value $s_n >2$.

Let $\varepsilon = (s_n - 2)/3$. Standard results on approximation of convex sets by polytopes show that there exists a polytope $Q$ such that $P \subset Q \subset (1+\varepsilon) P$.  In particular, $\diam_{\rho} (Q) \le (1+\varepsilon)\diam_{\rho} (P) = 2(1+\varepsilon) < s_n$.

We define $\mathcal{G} \subseteq \mathcal{F}$ to be the family of half-spaces in $\mathcal F$ whose bounding hyperplanes are parallel to some facet of $Q$, scaled by a factor of $1/s_n$.  The intersection of every $n$ or fewer sets in $\mathcal{G}$ has $\rho$-diameter greater than or equal to $1$. But $\bigcap \mathcal{G} \subseteq (1/s_n)Q$, so its $\rho$-diameter is strictly less than $1$.
\end{proof}



\section{Diameter results for \texorpdfstring{$2d$}{2d}-tuples}\label{sec:frac}

We combine Theorem \ref{thm:v-width} with volume concentration properties of balls to prove theorems \ref{thm:diameter-fractional} and \ref{thm:diameter-fractional-colorful}.  The properties described below can be found in Keith Ball's expository notes \cite{ball1997elementary}.

Let $B$ be a ball centered at the origin, $c > 0$, and $u$ be a unit vector.  The \textit{$c$-cap} of $B$ in the direction $u$ is
\[
C(B,c, u) = \{x \in B : \langle x, u \rangle \ge c \}.
\]
For two unit vectors $u$ and $v$, we have that $v \in C(B, c, u)$ if and only if $u \in C(B, c, v)$.

Let $B_d$ be the unit ball in $\rr^d$, and let $r_d$ be the radius of a volume-one ball in $\rr^d$. Asymptotically, $r_d \sim  d^{1/2}/\sqrt{2 \pi e}$. For a fixed unit vector $u$ and real number $x$, the $(d-1)$-dimensional volume of the intersection of $r_dB_d$ with the hyperplane $\{y \in \rr^d : \langle u, y \rangle = x\}$ converges to $\sqrt{e} \exp (-\pi e x^2)$ as $d \to \infty$. In other words, the volume of the region $\{y \in \rr^d: |\langle y, u \rangle|< x \}\cap r_dB_d$ converges as $d \to \infty$, and converges to zero as $x \to 0$.  For any fixed constant $c$, we define
\begin{align}
    \gamma(c) &= \inf_{d \ge 2} \operatorname{vol}\left[ C\!\left(r_d B_d, \frac{c r_d}{\sqrt{d}}, u\right) \cup\, C\!\left(r_d B_d, \frac{c r_d}{\sqrt{d}}, - u\right)\right] \\
    & = \frac{1}{\operatorname{vol}{(B_d)}}\inf_{d \ge 2} \operatorname{vol}\, \left[ C\!\left( B_d, \frac{c}{\sqrt{d}}, u\right) \cup\, C\!\left(B_d, \frac{c}{\sqrt{d}}, - u\right)\right]\label{eq:slab-size}
\end{align}
This is the remaining volume of $r_d B_d$ after removing a slab centered at the origin with width approximately $2c/\sqrt{2\pi e}$ (see Figure \ref{fig:volume-slices}).  From the discussion above, $\gamma(c) \to 1$ as $c \to 0$. Equation \eqref{eq:slab-size} shows that $\gamma(c)$ is the fraction of the volume of two opposite $(c d^{-1/2})$-caps in the unit sphere. If $c < \sqrt{2}$, the volume of $C\big(r_dB_d, cr_d d^{-1/2}, u\big)$ is strictly positive for each $d \geq 2$ and tends to a positive limit as $d \to \infty$. So $\gamma(c) > 0$ for every $c \in (0,\sqrt{2})$.

\begin{figure}
	\centerline{\includegraphics[scale=0.9]{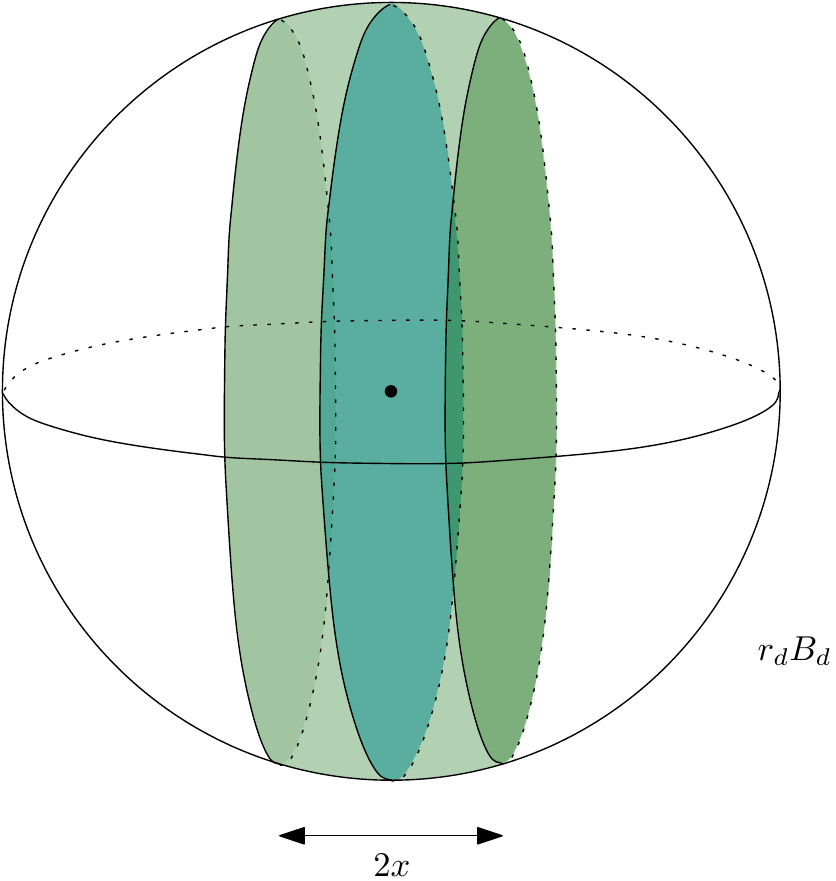}}
	\caption{The volume of the region in $r_d B_d$ between two parallel hyperplanes at distance $x$ from the origin converges as $d \to \infty$.}
	\label{fig:volume-slices}
\end{figure}

\begin{proof}[Proof of Theorem \ref{thm:diameter-fractional}]
	Let $n = \lvert \ff \rvert$. For each subfamily $\ff' \subseteq \ff$ of $2d$ sets whose intersection has diameter greater than or equal to 1, assign a unit vector $u_{\ff'}$ such that $\bigcap \ff'$ contains a unit segment with direction $u_{\ff'}$.  Let $\mathcal G$ be the collection of sets
	\[
    	C\!\left({B}_d, c d^{-1/2}, u_{\ff'}\right) \cup\, C\!\left({B}_d, c d^{-1/2}, -u_{\ff'}\right)
	\]
	 where $\ff'$ is a $2d$-tuple of $\ff$ whose intersection has diameter greater than or equal to 1.  Each such set covers at least a $\gamma(c)$ fraction of the volume of the unit ball.  Therefore, the total volume covered amongst all sets in $\mathcal G$ is at least
	\[
    	 \gamma(c)\cdot \alpha \binom{n}{2d}.
	\]
	Since $\bigcup \mathcal G$ does not contain the origin, there is a nonzero point $x$ in the unit ball covered at least $\gamma(c)\cdot \alpha \binom{n}{2d}$ times by $\mathcal G$.  Setting $v = x/||x||$, the set
	\[
    	C\!\left({B}_d, c d^{-1/2}, v\right) \cup\, C\!\left({B}_d, c d^{-1/2}, -v\right)
	\]
	contains at least $\gamma(c)\cdot \alpha \binom{n}{2d}$ different vectors $u_{\ff'}$.  Thus, the intersection of at least $\gamma(c)\cdot \alpha \binom{n}{2d}$ different $2d$-tuples of $\ff$ have $v$-width greater than or equal to $cd^{-1/2}$. An application of \cref{thm:v-width-fractional} finishes the proof, using the optimal bound for $\beta$ due to Kalai \cites{Kalai:1984isa, Kalai:1986ho}.	
	\end{proof}
	
\begin{proof}[Proof of Theorem \ref{thm:diameter-fractional-colorful}]
    The proof is analogous to that of Theorem \ref{thm:diameter-fractional}. In place of the fractional Helly theorem, we apply the colorful fractional theorem of B\'ar\'any, Fodor, Montejano, Oliveros, and P\'or \cite{Barany:2014bp} with the bound for $\beta$ by Kim \cite{Kim:2017by}.
%
	\end{proof}

\section{Acknowlegments}
The authors thank Silouanos Brazitikos and Gennadiy Averkov for their comments on this work.

\end{document}